\documentclass[11pt,leqno]{article}
\usepackage[utf8]{inputenc}
\usepackage{amssymb} 
\usepackage{amstext}
\usepackage{indentfirst}
\usepackage{amsthm}
\usepackage{fullpage}
\usepackage{oldgerm}
\usepackage{amsmath}
\usepackage{xcolor}
\usepackage{hyperref}

\usepackage{mathrsfs}
\numberwithin{equation}{section}

\usepackage{commath} 
\usepackage[mathscr]{euscript} 
\usepackage{comment} 
\usepackage{mathtools} 
\usepackage{verbatim} 
\usepackage{physics} 

\usepackage{wrapfig}
\usepackage{mwe}
\usepackage{epsfig}
\usepackage{color}
\usepackage{graphicx}
\usepackage{amsfonts,amssymb,amsmath}

\newtheorem{thm}{Theorem}

\newtheorem{lem}{Lemma}

\newtheorem{prop}{Proposition}
\newtheorem{rem}{Remark}

\def\eps{\varepsilon}
\def\phi{\varphi}

\newcommand{\RM}{\mathbb{R}}

\usepackage{epsfig}

\title{Orbital Stability of Periodic Traveling
Waves in the $b$-Camassa-Holm Equation}
\author{Brett Ehrman\thanks{Department of Mathematics, University of Kansas, 1460 Jayhawk Blvd., Lawrence, KS 66045, USA; \texttt{ehrman.brett@ku.edu}}
~~ \& ~~
Mathew A. Johnson\thanks{Department of Mathematics, University of Kansas, 1460 Jayhawk Blvd., Lawrence, KS 66045, USA; \texttt{matjohn@ku.edu}}
}
\date{\today}

\begin{document}
\maketitle

\begin{abstract}
In this paper, we identify criteria that guarantees the nonlinear orbital stability of a given periodic traveling wave solution 
within the b-family Camassa-Holm equation.   These periodic waves exist as 3-parameter families (up to spatial translations) of smooth
traveling wave solutions, and their stability criteria are expressed in terms of Jacobians of the conserved quantities
with respect to these parameters.  The stability criteria utilizes a general Hamiltonian structure which exists for every $b>1$, and hence
applies outside of the completely integrable cases ($b=2$ and $b=3$).
\end{abstract}

{\begin{center}
\end{center}
}


\section{Introduction}

We study the nonlinear stability of periodic traveling wave solutions for the b-family of Camassa-Holm equations (b-CH), which is given by
\begin{equation}\label{e:bch}
u_t - u_{txx} + (b+1)uu_x = bu_xu_{xx} + uu_{xxx}
\end{equation}
where here $b\in\RM$ is a parameter.   The family of models \eqref{e:bch} was introduced in \cite{DKK02,DGH01} by using transformations
of the integrable hierarchy of KdV equations.  In the modeling, the b-CH equation describes the horizontal velocity $u=u(x,t)$ for the unidirectional
propagation of water waves on a free surface in shallow water over a flat bed.  In the special cases $b=2$ and $b=3$ it is known that \eqref{e:bch}
is completely integrable via the inverse scattering transform, with $b=2$ corresponding to the well-studied Camassa-Holm equation and $b=3$
corresponding to the Degasperis-Procesi equation.  Further, it is known according to various tests for integrability
that the b-CH equation fails to be integrable outside of the cases $b=2$ and $b=3$: see, for example, \cite{MN02,Hone09}.

Besides being completely integrable, both the Camassa-Holm and Degasperis-Procesi equations have received a considerable amount of attention
due to the fact that they admit both smooth and peaked traveling solitary and periodic waves, as well as multi-soliton type solutions.  Additionally,
in these integrable cases the equation \eqref{e:bch} admits multiple Hamiltonian structures: the Camassa-Holm equation (corresponding to $b=2$) admits
three separate Hamiltonian structures, while the Degasperis-Procesi equation admits two.  Concerning the stability of smooth periodic and solitary waves,
there have been multiple studies of their orbital stability by working with the conserved energy integrals in the natural energy space:
see, for example, \cite{CW02,LLW20,GMNP22,GP22} and references therein.  Unfortunately,
the Hamiltonian structures used in these studies are a special feature due to the completely integrability of \eqref{e:bch} in the cases $b=2$ and $b=3$,
and hence these results cannot be extended to more general values of $b$.

In this work, we are interested in developing an orbital stability theory for periodic traveling wave solutions of \eqref{e:bch} which applies to any $b>1$.  
Note that the analogous work for smooth solitary wave solutions of \eqref{e:bch} was recently carried out in \cite{LP22,LL23}, where the authors utilized
the Hamiltonian formulation of the b-CH equation from \cite{DHH03}.  This particular Hamiltonian formulation is expressed in terms of the so-called 
momentum density $m=u-u_{xx}$ of the solution, and applies to the entire family of equations \eqref{e:bch} outside of the case $b=1$.  In this work,
we extend the work from \cite{LP22} to the case of periodic traveling waves of \eqref{e:bch}, which inherently constitutes a much larger
family of solutions: periodic traveling waves constitute (modulo spatial translations) a 3-parameter family of solutions, while the solitary waves only constitute
(again modulo spatial translations) a 1-parameter family.
The higher dimensionality of the associated solution manifold introduces a number of technical challenges not encountered in the solitary wave study.
We approach these challenges by using a methodology introduced in \cite{J09} in the stability analysis of periodic traveling wave solutions
of the generalized KdV equations.

We also note that there have been a series of recent works on the orbital stability of periodic traveling wave solutions
of \eqref{e:bch} in the completely integrable cases $b=2$ and $b=3$: see \cite{GMNP22} and \cite{GP22}.  While there are obviously some
similarities between these and the present work, the methodology for handling the higher dimensionality of the manifold of periodic
traveling waves are quite different, and we will expand upon this throughout this manuscript.  Additionally, both the works \cite{GMNP22,GP22} utilize
Hamiltonian structures, and the conserved quantities associated to them, which do not extend to general $b>1$.

The basic approach utilized in this work is by now classical, essentially being an application of the methodology formalized
by Grillakis, Shatah and Strauss in \cite{GSS1} for the stability of nonlinear solitary waves in Hamiltonian systems.  
Basically, we start off by carefully studying the existence theory for periodic traveling wave solutions $u(x,t)=\phi(x-ct)$, of period $T$ say,
of \eqref{e:bch}.  As the Hamiltonian structure associated with \eqref{e:bch} is expressed solely in terms of the 
momentum density $m=u-u_{xx}$ of solutions of \eqref{e:bch}, we then encode the momentum density of our solution $\mu=\phi-\phi''$
as a critical point of an appropriate action functional built entirely out of conserved quantities for the b-CH flow.    Using
Taylor series it quickly becomes apparent that the (orbital) stability or instability of $\mu$ is intimately related to
the spectral properties of the second variation of the associated action functional evaluated at $\mu$.  We identify conditions
on the underlying wave $T$-periodic wave $\mu$ that guarantees this second variation has exactly one negative $T$-periodic eigenvalue, a simple $T$-periodic
eigenvalue at the origin (associated to the translational invariance of \eqref{e:bch}), and the rest of the  $T$-periodic eigenvalues are strictly positive and bounded away from zero.
By now considering the notion of orbtial stability (i.e. identifying functions up to spatial translations) and the class of perturbations appropriately, we then 
identify a final set of conditions guaranteeing the stability of $\mu$, and hence $\phi$, to $T$-periodic perturbations.

\

The outline of our paper is as follows.  In Section \ref{S:basics} we review some basic features regarding the b-CH equation \eqref{e:bch}, including
the Hamiltonian structure used in this work as well as the existence theory for periodic traveling wave solution of \eqref{e:bch}.
Our main stability analysis is contained in Section \ref{S:stab_criteria}, culminating into our main result, Theorem \ref{T:main}. 
Finally, in Appendix \ref{S:Appendix} we establish a technical result used in Section \ref{S:stab_criteria}.

%
%
%
%
%

\

\noindent
{\bf Acknowledgments:} The authors would like to thanks Stephane Lafortune and Dmitry Pelinovsky for several helpful discussions regarding the b-CH equations.
The work of both authors was partially funded by the NSF under grant DMS-2108749.  MAJ was also supported by the Simons Foundation Collaboration grant 
number 714021.  The authors also thank the referees for their thoughtful and helpful feedback, which we believe greatly improved the paper.

\section{Some Basic Properties of the b-CH Family}\label{S:basics}

In this section, we collect some basic results regarding the b-CH equation \eqref{e:bch} and its solutions.

\subsection{Hamiltonian Structure and Conservation Laws}

For each $b\in\RM$ the b-CH equation \eqref{e:bch} is known to be a Hamiltonian system in terms of the so-called momentum density
$m=u-u_{xx}$.  Since we are interested in the local dynamics about periodic traveling waves, here we restrict our 
discussion of the Hamiltonian formulation on the space $L^2_{\rm per}(0,T)$ for some $T>0$.  To this end, it is straightforward to see that \eqref{e:bch} can be rewritten as
\begin{equation}\label{e:bch_m}
m_t+um_x+bmu_x=0
\end{equation}
which, for $b\neq 1$, admits the Hamiltonian formulation
\begin{equation}\label{e:Ham_formulation}
\frac{dm}{dt} = J_m\frac{\delta E}{\delta m}
\end{equation}
where here 
\begin{equation}\label{e:J}
J_m := \frac{-1}{b-1}(bm\partial_x + m_x)(1- \partial_x^2)^{-1}\partial_x^{-1}(b\partial_x m - m_x)
\end{equation}
is a (state-dependent) skew-adjoint operator on $L^2_{\rm per}(0,T)$ and
\begin{equation}\label{e:E}
E(m) := \int_0^T m~dx.
\end{equation}
We note that from \eqref{e:bch_m} the functional $E(m)$ is readily seen to be a conserved quantity, often referred to as either the total momentum
or the mass of the wave.  Additionally,
it is known that for general $b\neq 0$ the b-CH equation admits two additional conserved quantities given by
\begin{equation}\label{e:F1F2}
F_1(m):=\int_0^T m^{1/b}dx~~~{\rm and}~~~F_2(m):=\int_0^T\left(\frac{m_x^2}{b^2m^2}+1\right)m^{-1/b}dx.
\end{equation}
We note that, in general, the above functionals are well-defined and smooth functions on the set
\[
X_T:=\left\{m\in H^1_{\rm per}(0,T):m(x)>0~~{\rm for~all}~x\in\RM\right\}.
\]
Assuming that \eqref{e:bch} is well-posed on some appropriate subset of $H^3_{\rm per}(0,T)$, our aim is to establish criteria for
the nonlinear orbital stability of $T$-periodic traveling wave solutions of the b-CH equation with respect to perturbations in $H^3_{\rm per}(0,T)$.
Regarding well-posedness, for general $b>1$ the initial-value problem for \eqref{e:bch} is locally well posed in the space $H^s_{\rm per}(0,T)$ for any $s>3/2$:
see \cite{EY08}.  Further, the initial-value problem with $b>1$ is ill-posed in $H^s_{\rm per}(0,T)$ for $s<3/2$ due to lack of continuous dependence and norm inflation, and hence
orbital stability in the energy space is only conditional with respect to the existence of local solutions.  For more details, see \cite{HGH16,LP22}.

Finally, a number of Jacobian determinants will arise throughout our work.  For notational simplicity, we adopt the following
notation for $2\times 2$ Jacobians 
\begin{equation}\label{e:Jacobian}
\{f,g\}_{x,y}:= \det\left(\frac{\partial(f,g)}{\partial(x,y)}\right)=
\det\left(\begin{array}{cc}
f_{x} & f_{y}\\
g_{x} & g_{y}
\end{array}\right)
\end{equation}
and similarly $\{f,g,h\}_{x,y,z}$ for the analogous $3\times 3$ Jacobian.

\subsection{Existence of Periodic Traveling Waves}\label{S:exist}

We now  study the existence of smooth periodic traveling wave solutions of \eqref{e:bch}.  We note that the existence theory for smooth solitary wave solutions
has been worked out in detail in \cite{LP22}.  The difference here, of course, is that the profile now does not have a constant asymptotic state as $x\to\pm\infty$.
Throughout, we will assume $b>1$.

Traveling wave solutions of \eqref{e:bch} correspond to solutions of the form $u(x,t) = \phi(x-ct)$ for some wave profile $\phi$ and wave speed $c>0$.
The profile $\phi(z)$ is thus required to be a stationary solution of the evolutionary equation
\begin{equation}\label{e:bch2}
u_t-u_{tzz}-cu_z+cu_{zzz}+(b+1)uu_z=bu_zu_{zz}+uu_{zzz}
\end{equation}
written here in the traveling coordinate $z=x-ct$.  After some rearranging, it follows that the such stationary solutions satisfy
the ODE
\begin{equation}\label{e:profile1}
(\phi-c)(\phi-\phi'')'+b\phi'(\phi-\phi'')=0,
\end{equation}
where here $'$ denotes differentiation with respect to $z$.  Note that, by elementary bootstrapping arguments, if $\phi$ is a $T$-periodic weak solution of \eqref{e:profile1} 
then we necessarily have $\phi\in C^\infty_{\rm per}(0,T)$ provided that either $\phi(x)<c$ or $\phi(x)>c$ for all $x\in\RM$.
Throughout our analysis, we consider only solutions satisfying 
\begin{equation}\label{e:sign}
\phi(x)<c~~{\rm for~all}~~x\in\RM.
\end{equation}
With this condition in mind, we note that, by multiplying both sides of  \eqref{e:profile1} by the integrating 
factor $(c-\phi)^b$, the ODE \eqref{e:profile1} can be rewritten as
\begin{equation}\label{e:profile2}
\frac{d}{dx}\left((c-\phi)^b(\phi-\phi'')\right)=0,~~{\rm i.e.}~~\phi-\phi''=\frac{a}{(c-\phi)^b}
\end{equation}
where here $a\in\RM$ is a constant of integration.  After another integration, the above may be further reduced by quadrature to 
\begin{equation}\label{e:quad}
\frac{1}{2}\left(\phi'\right)^2=E-\left(-\frac{1}{2}\phi^2+\frac{a}{(b-1)(c-\phi)^{b-1}}\right),
\end{equation}
where here $E\in\RM$ is another constant of integration,
and hence the existence and non-existence of bounded solutions of \eqref{e:profile1} can be determined 
by studying the potential function
\begin{equation}\label{e:potential}
V(\phi;a,c):=-\frac{1}{2}\phi^2+\frac{a}{(b-1)(c-\phi)^{b-1}}
\end{equation}
Indeed, by standard phase-plane analysis it is clear that a necessary and sufficient condition for the existence of periodic solutions of 
\eqref{e:profile1} is that the potential $V(\cdot;a,c)$ have a local minimum.  

\begin{rem}\label{r:quad_alternate}
It is also possible to integrate the third order profile equation \eqref{e:profile1} directly, yielding
\[
(c-\phi)(\phi-\phi'')+\frac{b-1}{2}\left((\phi')^2-\phi^2\right)=f
\]
or, equivalently,
\[
-(c-\phi)\phi''+c\phi+\frac{b-1}{2}\left(\phi'\right)^2-\frac{b+1}{2}\phi^2=f
\]
for some constant $f\in\RM$.  Using \eqref{e:profile2} it immediately follows that $f=(b-1)E$, with $E$ being as in \eqref{e:quad}, and hence, in this sense, the above equation
is redundant from \eqref{e:profile2}-\eqref{e:quad}.
\end{rem}

\begin{figure}[t]
\begin{center}
\includegraphics[scale=.8]{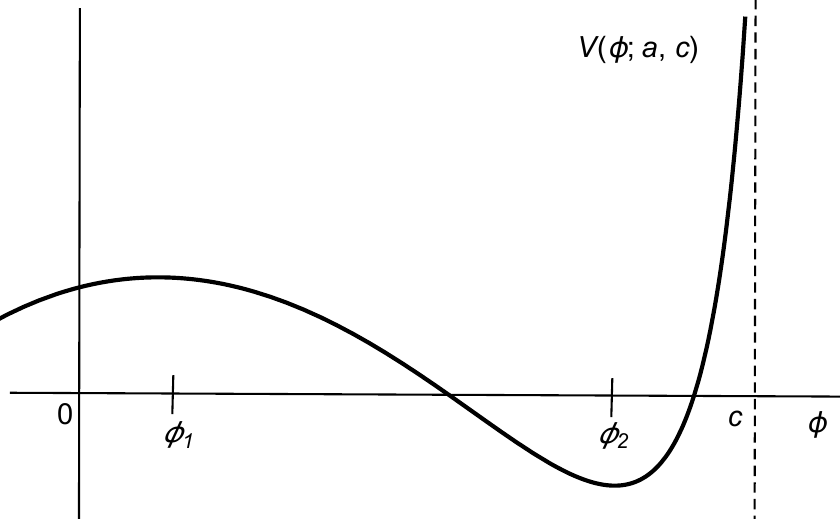}
\caption{Depiction of the effective potential $V(\phi;a,c)$ for an admissible value of $a$.  Note there is a vertical asymptote at $\phi=c$ and that, and that
all the periodic solutions here exist for $\phi<c$.}
\label{f:Vplot} 
\end{center}
\end{figure}

To study the critical points of $V$ satisfying \eqref{e:sign} we note that
\[
V_\phi(\phi;a,c)=-\phi+\frac{a}{(c-\phi)^b}
\]
and hence seeking critical points of $V(\cdot;a,c)$ with $\phi<c$ for fixed parameters $(a,c)$ is equivalent to seeking roots of the function
\[
g(\phi):=\phi(c-\phi)^b-a,~~\phi<c.
\]
To this end, note that 
\[
g'(\phi)=(c-\phi)^{b-1}\left(c-(b+1)\phi\right)
\]
and hence the only critical point of $g$ satisfying $\phi<c$ is $\phi=\frac{c}{b+1}$.  In particular, elementary calculations show that
\[
g''\left(\frac{c}{b+1}\right)=-bc\left(\frac{bc}{b+1}\right)^{b-2}<0
\]
and hence $\phi=\frac{c}{b+1}$ is a strict local maximum of $g$.  Since $g\in C^\infty(-\infty,c)$ it follows by above that $g$ is strictly
increasing on $\left(-\infty,\frac{c}{b+1}\right)$ and strictly decreasing on $\left(\frac{c}{b+1},c\right)$.  Further, we have
\[
g(0)=-a,~~g\left(\frac{c}{b+1}\right)=\frac{b^b c^{b+1}}{(b+1)^{b+1}}-a,~~\lim_{\phi\to c}g(\phi)=-a
\]
and hence, by the Intermediate Value Theorem, for each $a\in\left(0,\frac{b^b c^{b+1}}{(b+1)^{b+1}}\right)$ the equation
$g(\phi)=0$ has exactly two solutions 
\begin{equation}\label{e:maxmin}
\phi_1\in\left(0,\frac{c}{b+1}\right)~~{\rm and}~~\phi_2\in\left(\frac{c}{b+1},c\right).
\end{equation}
For such $a$ we necessarily have $V'(\phi_1)=V'(\phi_2)=0$ and, since
\[
V_\phi(\phi;a,c)=-\frac{g(\phi)}{(c-\phi)^b},
\]
it follows that $V$ achieves unique local max and min values at $\phi_1$ and $\phi_2$, respectively. 

It follows by the above and elementary phase plane analysis that if we define the set
\begin{equation}\label{e:ExistenceSet}
\mathcal{B}:=\left\{(a,E,c)\in\RM^3:c>0,~~a\in\left(0,\frac{b^b c^{b+1}}{(b+1)^{(b+1)}}\right),~~E\in\left(V(\phi_2;a,c),V(\phi_1;a,c)\right)\right\}
\end{equation}
then for each set of parameters $(a,E,c)\in\mathcal{B}$ the profile equation \eqref{e:profile1} admits a one-parameter family, parameterized by translation invariance,
of smooth periodic solutions $\phi(x;a,E,c)$ satisfying $\phi<c$ and with period
\[
T=T(a,E,c)=\sqrt{2}\int_{\phi_{\rm min}}^{\phi_{\rm max}}\frac{d\phi}{\sqrt{E-V(\phi;a,c)}},
\]
where here $\phi_{\rm min}$ and $\phi_{\rm max}$ denote, respectively, the minimum and maximum roots of the equation $E-V(\cdot;a,c)=0$, respectively, and hence
correspond to the minimum and maximum values of the corresponding periodic solution $\phi$.  Note that since the values $\phi_{\rm min}$ and $\phi_{\rm max}$
are smooth functions of the traveling wave parameters $(a,E,c)$, it follows that the period function $T(a,E,c)$ represents a $C^1$ function on $\mathcal{B}$.

Putting all of the above together, it follows that the b-CH equation \eqref{e:bch} admits a 4-parameter family, constituting a $C^1$ manifold, of periodic traveling wave solutions 
of the form
\[
\phi(x-ct+x_0;a,E,c),~~x_0\in\RM,~~(a,E,c)\in\mathcal{B}
\]
with period $T=T(a,E,c)$.  Recalling from the previous section that the Hamiltonian formulation for the b-CH equation is expressed entirely in terms
of the momentum density $m=u-u_{xx}$, we note from \eqref{e:profile2} that if $\phi(x;a,E,c)$ is a smooth $T$-periodic stationary solution \eqref{e:bch2}
as constructed above, then 
\begin{equation}\label{e:relation1}
\mu(x;a,E,c)=\frac{a}{(c-\phi(x;a,E,c))^b}
\end{equation}
is a smooth $T$-periodic stationary solution of\footnote{With a slight abuse of notation, henceforth we will use replace the traveling variable $z$ simply by $x$.}
\[
m_t-cm_x+um_x+bmu_x=0
\]
satisfying $\mu>0$ for all $x\in\RM$.  In particular, for each $(a,E,c)\in\mathcal{B}$ we have $\mu(x;a,E,c)\in X_T$ with $T=T(a,E,c)$.  Further, following
the procedure  above we can restrict the conserved quantities in \eqref{e:F1F2} to the manifold of periodic traveling wave solutions of the b-CH equation
yielding, with slight abuse of notation, $C^1$ functions $F_1,F_2:\mathcal{B}\to\RM$ defined via
\[
\left\{\begin{aligned}
F_1(a,E,c)&:= \int_0^{T(a,E,c)} \mu(x;a,E,c)^{1/b}dx\\
F_2(a,E,c)&:=\int_0^{T(a,E,c)}\left(\frac{\mu_x(x;a,E,c)^2}{b^2\mu(x;a,E,c)^2}+1\right) \mu(x;a,E,c)^{-1/b}dx
\end{aligned}\right.
\]
As we will see, the gradients of these conserved quantities along the manifold of periodic traveling wave solutions of \eqref{e:bch}, or equivalently \eqref{e:bch_m},
will plan a central role in our forthcoming analysis.

\begin{rem}\label{R:compare1}
As mentioned in the introduction, there are similarities between the present work and that in \cite{GMNP22,GP22}, where authors considered the
stability of periodic traveling wave solutions of the b-CH equation \eqref{e:bch} in the completley integrable cases $b=2$ and $b=3$.  In addition
to using a different Hamiltonian structure in our work (recall those used in \cite{GMNP22,GP22} do not extend outside the completely integrable cases),
our work differs in how we parameterize the set of periodic traveling waves.  In \cite{GMNP22,GP22}, the authors, for a fixed $b>1$, start by fixing
a wave speed $c>0$ and thus reducing to a 2-parameter family depending on $(a,E)$.  They then consider a curve $E=E(a)$ in parameter space where
the period is held constant, and thus reduce to a 1-parameter family of waves with a fixed period and fixed wave speed depending only
on the parameter $a$.  In our work, by contrast, we work with the full 4-parameter family of periodic traveling waves and their variations
with respect to all four parameters.  This results in the introduction of a number of Jacobian determinants associated with this parameterization arising
in our work, which of course are geometrically related to the ability to locally  reparameterize  the manifold of solutions.  This geometric/Jacobian
approach has had significant success not only in orbital stability analysis of periodic traveling waves in nonlinear dispersive systems, 
but also in the stability analysis of such solutions to more general classes of perturbations.  See, for example, \cite{BrJK14,BrJ10,BrJK11,HJ15,J09,J_BBM,JP20,JZ_KP,JZBR10}.
\end{rem}

We close this section by noting that additional families of smooth periodic traveling waves exist for $b\leq 1$, and the associated existence and stability
theory can be handled in a similar way as in this paper.  Further, when $b>1$ one can use the above analysis of the potential $V$ 
to additionally show that no smooth periodic traveling waves exist for $\phi>c$.

\section{Orbital Stability Criteria}\label{S:stab_criteria}

In this section, we derive conditions which guarantee a given $T$-periodic traveling wave solution of the b-CH equation is orbitally stable
with respect to $T$-periodic perturbations. Throughout our analysis, we will work in the momentum density formulation \eqref{e:bch_m}.
To this end, we begin by attempting to encode a given $T$-periodic traveling wave solution $\mu(\cdot;a,E,c)$ as a critical point
of an action functional of the form
\begin{equation}\label{e:LambdaOmega}
\Lambda(m) = E(m) - \omega_1F_1(m) - \omega_2F_2(m),~~\omega_1,\omega_2\in\RM
\end{equation}
defined when $m \in H_{per}^1[0,T]$ and $m >0$, where here $E$ denotes the total momentum functional \eqref{e:E} and $F_1,F_2$ are the conserved quantities
given in \eqref{e:F1F2}.  More specifically, we seek values $\omega_1,\omega_2\in\RM$ such that the profile equation \eqref{e:quad} is equivalent
to the Euler-Lagrange equation of the functional $\Lambda$.  This is accomplished in the following Lemma.

\begin{lem}\label{L:Lagrange}
For a fixed $b>1$, let $\mu(\cdot;a,E,c)$ be a $T$-periodic solution of \eqref{e:quad}.  Then $\mu$ is a critical point
of the action functional $\Lambda(m)$ provided that 
\[
\omega_1 = \frac{b-1}{2a^\frac{1}{b}}\left[2E + c^2\right] ~~{\rm and}~~ 
\omega_2 = \frac{1}{2}a^{\frac{1}{b}}(b-1)
\]
\end{lem}

\begin{proof}
The proof is given in \cite[Lemma 2]{LP22}.  For completeness, we outline the main idea of the argument.  First, from \eqref{e:F1F2} we note that
\[
\frac{\partial E}{\partial m}=1,~~~~\frac{\partial F_1}{\partial m}(m) = \left(\frac{1}{b}\right)m^{1/b-1}
\]
and
\[
\frac{\partial F_2}{\partial m}(m)= \frac{1}{b}m^{-(1/b+1)}\left(\frac{(2b+1)m_x^2}{b^2m^2}-\frac{2m_{xx}}{bm}-1\right).
\]
and hence the equation $\frac{\partial\Lambda}{\partial m}(\mu)=0$ is equivalent to the differential equation
\begin{equation}\label{e:LambdaPrime}
1 - \omega_1\left(\frac{1}{b}\right) \mu^{\frac{1}{b}-1} - \omega_2\frac{1}{b}\mu^{\frac{-1}{b}-1}
\left(\frac{(2b+1)\mu_x^2}{b^2\mu^2}-\frac{2\mu_{xx}}{b\mu}-1\right)=0
\end{equation}
As in \cite{LP22}, we note that if $\mu$ is a $T$-periodic weak solution of \eqref{e:LambdaPrime} then by elementary bootstrapping
arguments we have $\mu\in C^\infty_{\rm per}(0,T)$.  Consequently, to show the critical points of $\Lambda$ are precisely the periodic
traveling wave solutions $\mu(\cdot;a,E,c)$ constructed in the previous section, it is sufficient to establish an equivalence between 
the differential equation \eqref{e:LambdaPrime} and \eqref{e:profile2}.

To this end, it is important to observe that differentiating \eqref{e:profile2} one has the relations
\begin{equation}\label{e:mu_derivatives}
\mu'=\frac{b\phi'}{c-\phi}\mu~~{\rm and}~~\mu''= \frac{b\phi''}{c-\phi}\mu + \frac{b(b+1)\left(\phi'\right)^2}{(c-\phi)^2}\mu
\end{equation}
and hence derivatives of $\mu$ in \eqref{e:LambdaPrime} can be replaced by derivatives of $\phi$ and multiples of $\mu$.  Upon making
these substitutions into \eqref{e:LambdaPrime} one finds by straightforward calculations that the resulting equation is equivalent to 
\eqref{e:quad} provided the stated choices of $\omega_1$ and $\omega_2$ are made.  For more details, see \cite{LP22}.
\end{proof}

\begin{rem}\label{r:LagrangeMult}
It is important to note that the Lagrange multipliers $\omega_1$ and $\omega_2$ found above are smooth functions of 
the traveling wave parameters $(a,E,c)$ on the entire existence set $\mathcal{B}$ defined in \eqref{e:ExistenceSet}.  In particular, we note that
\[
\nabla_{(a,E,c)}\omega_1=\left<\frac{-(b-1)}{2ba^{\frac{1}{b}+1}}\left[2E + c^2\right],~\frac{b-1}{a^\frac{1}{b}},~\frac{c(b-1)}{a^{1/b}}\right>
\]
and
\[
\nabla_{(a,E,c)}\omega_2=\left<\frac{1}{2b}a^{\frac{1}{b}-1}\left(b-1\right),~0,~0\right>.
\]
In particular, while $\omega_2$ depends only on the parameter $a$, the Lagrange multiplier $\omega_1$ depends on all three traveling
wave parameters $(a,E,c)$.  This is a quite different case than that encountered, say, in the KdV, BBM or NLS type equations where each Lagrange multiplier 
depends on only one of the traveling wave parameters: see, for instance, \cite{J09,J_BBM,GH07,LBrJM21}. In the forthcoming analysis, it is interesting to track the effect of this additional
dependence of the Lagrange multipliers.  
\end{rem}

By Lemma \ref{L:Lagrange}, our periodic traveling wave solutions $\mu(\cdot;a,E,c)$ of the b-CH equation are realized as critical points of the action
functional $\Lambda$.  In order to classify this critical point as a local minimum, maximum or a saddle point, we study the second
derivative 
\[
\mathcal{L}[\mu]:=\frac{\delta^2\Lambda}{\delta m^2}(\mu)
\]
evaluated at the wave $\mu(\cdot;a,E,c)$.  To this end, we note through straightforward, but lengthy and tedious, calculations (see, for example,
\cite[Corollary 1]{LP22}) the operator $\mathcal{L}[\mu]$ can be expressed as a Sturm-Liouville operator of the form
\begin{equation}\label{e:L}
\mathcal{L}[\mu]=-p(x)\partial_x^2-q(x)\partial_x-r(x)
\end{equation}
where $p$, $q$, and $r$ are smooth, $T$-periodic functions with\footnote{The explicit expression for $r$ is quite lengthy, and also irrelevant to our 
calculations.  As such, we do not include it.}
\begin{equation}\label{e:L_coeff}
p(x) = \frac{\omega_2}{b^2\mu^{2+1/b}},~~~q(x)=-\frac{\omega_2(2b+1)\mu_x}{b\mu^{3+1/b}}.
\end{equation}
In particular, we note, since $b>1$, that $p(x)>0$ for all $x\in\RM$.  

It follows that $\mathcal{L}[\mu]$
is a self-adjoint linear operator acting on $L^2_{\rm per}(0,T)$ with compactly embedded domain $H^2_{\rm per}(0,T)$.  Consequently,
it is well-known\footnote{For example, see \cite[Section 2.3]{KP13}.} that the spectrum of $\mathcal{L}[\mu]$ on $L^2_{\rm per}(0,T)$ consists 
of an increasing sequence of real eigenvalues satisfying
\[
\lambda_0<\lambda_1\leq\lambda_2<\lambda_3\leq\lambda_4\ldots \nearrow +\infty
\]
and that the associated eigenfunctions $\{\psi_n\}_{n=0}^\infty$ forms an orthogonal basis for $L^2_{\rm per}(0,T)$.  Further, 
the (ground state) eigenfunction $\psi_0$ can be chosen to be strictly positive, while for each $n\in\mathbb{N}$ the eigenfunctions $\psi_{2n-1}$ 
and $\psi_{2n}$ have precisely $2n$ simple zeroes on $[0,T)$.

Noting that, by the spatial translation invariance of \eqref{e:bch_m}, we have
\[
\mathcal{L}[\mu]\mu_x=0
\]
it follows that $\lambda=0$ is a $T$-periodic eigenvalue of $\mathcal{L}[\mu]$.  Since $\mu_x$ has precisely two roots on $[0,T)$, by construction, it further
follows that $\lambda=0$ is either the second or third smallest eigenvalue of $\mathcal{L}[\mu]$ and hence, in particular, $\mathcal{L}[\mu]$ has at least 
one negative eigenvalue.  A precise count on the number of negative eigenvalues, as well as the simplicity of the zero eigenvalue, is established
in the following result.

\begin{thm}\label{T:morse}
The spectrum of the linear operator $\mathcal{L}[\mu]$ considered on the space $L^2_{\rm per}(0,T)$ satisfies the following 
trichotomy\footnote{Recall that the notation $\{f,g\}_{x,y}$ is defined in \eqref{e:Jacobian} above.}
\begin{itemize}
\item[(i)]  If $\{T,\omega_1\}_{E,c}>0$ then $\mathcal{L}[\mu]$ has exactly one negative eigenvalue, a simple eigenvalue at zero, and the rest of the
spectrum is strictly positive and bounded away from zero.
\item[(ii)]  If $\{T,\omega_1\}_{E,c}=0$ then $\mathcal{L}[\mu]$ has exactly one negative eigenvalue, a double eigenvalue at zero, and the rest of the
spectrum is strictly positive and bounded away from zero.
\item[(iii)] If $\{T,\omega_1\}_{E,c}<0$ then $\mathcal{L}[\mu]$ has exactly two negative eigenvalues, a simple eigenvalue at zero, and the rest of the
spectrum is strictly positive and bounded away from zero.
\end{itemize}
\end{thm}

\begin{rem}\label{R:morse}
The sign of the quantity $\{T,\omega_1\}_{E,c}$ can be analytically determined to be positive for at least some waves in the integrable cases $b=2$ and $b=3$, corresponding
to the classical Camassa-Holm equation and the Degasperis-Procesi equations, respectively.  This is discussed in detail in Appendix \ref{S:A1}.  While the forthcoming
stability analysis requires $\{T,\omega_1\}_{E,c}>0$, it is an interesting open question to determine whether failure of this condition is associated with
instability of the wave, or if stability is still possible in this case.
\end{rem}

The general strategy for the proof is similar to that given in \cite[Theorem 4]{GMNP22}, and relies on a well-known result from Floquet theory as well
as Sylvester's Inertial Law theorem.  Before we proceed with the proof of Theorem \ref{T:morse}, we state these auxiliary results.

\begin{lem}\label{L:Floquet}\cite{N09}
Consider the Schr\"odinger operator $\mathcal{M}=-\partial_x^2+Q(x)$ with an even, $T$-periodic, smooth potential $Q$.  
Assume that there exists linearly independent functions $\phi_1,\phi_2\in L^2_{\rm per}(0,T)$ that are solutions of $\mathcal{M}w=0$
and such that there exists constant $\theta\in\RM$ such that
\[
\phi_1(x+T) = \phi_1(x) + \theta \phi_2(x)~~{\rm and}~~
\phi_2(x+T) = \phi_2(x).
\]
Further,  suppose that $\phi_2$ has two zeros on $[0,T)$. The zero eigenvalue of $\mathcal{M}$ is simple if $\theta\neq 0$ and double if $\theta=0$.
Furthermore, $\mathcal{M}$ has one negative eigenvalue if $\theta\geq 0$ and two negative eigenvalues if $\theta<0$.
\end{lem}

\begin{lem}[Sylvester's Inertial Law \cite{L03}]\label{L:Sylvester}
Let $L$ be a self-adjoint operator on a Hilbert space $H$, and let $\mathcal{S}$ be a bounded, invertible operator on $H$.  Then the operators $L$ and $\mathcal{S}L\mathcal{S}^*$
have the same inertia, i.e. the dimensions of the negative, null, and positive invariant subspaces of $H$ for these two operators are the same.
\end{lem}

With these results in hand, we now establish Theorem \ref{T:morse}.

\begin{proof}[Proof of Theorem \ref{T:morse}]
The basic strategy, which again is similar to that in \cite{GMNP22}, is to first show that the Sturm-Liouville operator $\mathcal{L}[\mu]$ 
can be written as
\[
\mathcal{L}[\mu]=\mathcal{S}\mathcal{M}\mathcal{S}^*
\]
for some linear Schr\"odinger operator $\mathcal{M}$ (as in Lemma \ref{L:Floquet} above) and some bounded, invertible operator $\mathcal{S}$ (as in
Lemma \ref{L:Sylvester} above).  To this end, first note by \eqref{e:L} the spectral problem $\mathcal{L}[\mu]v=\lambda v$ can be written as
\[
p(x)v''+q(x)v'+\left(r(x)+\lambda\right)=0.
\]
We now define the function
\[
D(x)=\int_0^x\frac{q(s)}{p(s)}ds=\ln\left(\left(\frac{c-\phi(x)}{c-\phi(0)}\right)^{-(2b+1)}\right)
\]
and making the change of variables
\begin{equation}\label{e:cov}
v(x) = w(x) e^{-D(x)/(4b+2)}
\end{equation}
a straightforward calculation shows that 
\[
-w''+Q(x)w=\lambda(c-\phi)^{-1} w,
\] 
where here $Q$ is smooth, $T$-periodic and even\footnote{The function $Q(x)$ can be explicitly expressed in terms of the functions 
$p(x)$, $q(x)$, $r(x)$ and $\phi(x)$ above.  However, the expression is quite lengthy and unnecessary to our calculations and is hence omitted.}.  

Defining the linear Schr\"odinger operator $\mathcal{M}:=-\partial_x^2+Q(x)$ and 
defining the multiplication operator $\mathcal{S}:=(c-\phi)^{1/2}$, which is clearly bounded and invertible, it follows that 
$\lambda\in\RM$ is an eigenvalue of $\mathcal{L}[\mu]$ if and only if it is an eigenvalue of $\mathcal{S}\mathcal{M}\mathcal{S}^*$.
By Sylverster's inertial law theorem it follows that $\mathcal{S}\mathcal{M}\mathcal{S}^*$, and by extension $\mathcal{L}[\mu]$, necessarily
has the same inertia as the operator $\mathcal{M}$.  

It remains to determine the inertia of the operator $\mathcal{M}$.  To this end, we aim to build functions $\phi_1$ and $\phi_2$ as in Lemma \ref{L:Floquet}
by first building corresponding functions for the operator $\mathcal{L}[\mu]$.  First, notice that $\mathcal{L}[\mu]\mu_x=0$ and
\[
\mathcal{L}[\mu]\mu_E=\frac{\partial \omega_1}{\partial E}\frac{\partial F_1}{\partial m}(\mu),
	~~\mathcal{L}[\mu]\mu_c=\frac{\partial\omega_1}{\partial c}\frac{\partial F_1}{\partial m}(\mu),
\]
where the last two relations come from differentiating the profile equation $\Lambda'(\mu)=0$ with respect to the parameters $E$ and $c$, respectively.
It follows that the functions
\[
\mu_x~~~{\rm and}~~~\frac{\partial \omega_1}{\partial c}\mu_E-\frac{\partial \omega_1}{\partial E}\mu_c=\{\mu,\omega_1\}_{E,c}
\]
provide  two linearly independent solutions, the first being odd and the second even, of the equation $\mathcal{L}[\mu]v=0$.  
We now define the functions
\[
y_1(x) =\frac{\{\mu(x),\omega_1\}_{E,c}}{\{\mu_+,\omega_1\}_{E,c}},~~~y_2(x)=\frac{\mu_x(x)}{\mu_{xx}(0)},
\]
which are well-defined thanks to the technical result in Lemma \ref{L:technical} (see Appendix \ref{S:Appendix}) and are linearly independent by parity,
and note that $y_2$ is $T$-periodic while $y_1$ satisfies
\begin{equation}\label{e:y1_identity}
y_1(x+T)=y_1(x)+\theta y_2(x),~~\theta=y_1'(T).
\end{equation}
To see the above equality, note that differentiating the identity $\mu(x+T)=\mu(x)$ with respect to the parameters $E$ and $c$ gives
\[
\mu_E(x+T)-\mu_E(x) = -T_E\mu_x(x),~~{\rm and}~~\mu_c(x+T)-\mu_c(x)=-T_c\mu_x(x)
\]
and hence
\begin{equation}\label{e:y1_step1}
y_1(x+T)=y_1(x)+\left(\frac{\mu_x(x)}{\mu_{xx}(0)}\right)\left(-\mu_{xx}(0)\frac{\{T,\omega_1\}_{E,c}}{\{\mu_+,\omega_1\}_{E,c}}\right)
\end{equation}
Similarly, differentiating the relation $\mu_x(T)=0$ with respect to $E$ and $c$ and using the $T$-periodicity of $\mu(x)$ gives
\[
\mu_{Ex}(T)=\mu_{xx}(T)T_E=\mu_{xx}(0)T_E~~~{\rm and}~~~\mu_{cx}(T)=\mu_{xx}(0)T_c
\]
and hence
\[
\theta=-\mu_{xx}(0)\frac{\{T,\omega_1\}_{E,c}}{\{\mu_+,\omega_1\}_{E,c}},
\]
which, along with \eqref{e:y1_step1}, verifies \eqref{e:y1_identity}.  We further note that since Lemma \ref{L:technical} in Appendix \ref{S:Appendix} implies
that  $\mu_{xx}(0)<0$ and $\{\mu_+,\omega_1\}_{E,c}>0$, it follows that
\[
{\rm sign}\left(\theta\right)={\rm sign}\left(\{T,\omega_1\}_{E,c}\right).
\]
Of specific note, this shows that $y_1$ provides a second linearly independent solution of $\mathcal{L}[\mu]v=0$ if and only if $\{T,\omega_1\}_{E,c}=0$.  

Now, observing from that the change of variables \eqref{e:cov} that $v(0)=w(0)$ and $v'(0)=w'(0)$, we define
\[
\phi_1(x)=\left(\frac{c-\phi}{c-\phi\left(0\right)}\right)^{-1/2}y_1(x)~~{\rm and}~~\phi_2(x)=\left(\frac{c-\phi}{c-\phi\left(0\right)}\right)^{-1/2}y_2(x)
\]
and note that, by construction, $\phi_{1,2}$ are linearly independent solutions of the differential equation $\mathcal{M}w=0$ and satisfy
\[
\phi_1(x+T)=\phi_1(x)+\theta\phi_2(x),~~~\phi_2(x+T)=\phi_2(x)
\]
for all $x\in\RM$, where $\theta$ is as in \eqref{e:y1_identity}.  Since $\phi_2$ has precisely two roots in $[0,T)$ by construction,
it  follows from Lemma \ref{L:Floquet} that the zero eigenvalue of $\mathcal{M}$ is simple if and only if $\theta\neq 0$.  If
$\theta\neq 0$, then $\lambda=0$ is a simple eigenvalue of $\mathcal{M}$ and, again by Lemma \ref{L:Floquet},
$\mathcal{M}$ has either one or two negative eigenvalues depending on whether $\theta>0$ or $\theta<0$, respectively.
The proof is complete by recalling that, by Sylverster's inertial law theorem, the operators $\mathcal{M}$ and $\mathcal{S}\mathcal{M}\mathcal{S}^*$, and hence $\mathcal{L}[\mu]$
by the change of variables \eqref{e:cov}, have the same inertia.
\end{proof}

\begin{rem}
It is interesting to note that the function $y_1$ would be considerably simpler if the Lagrange multiplier $\omega_1$ depended only one one of the traveling wave parameters
$E$ or $c$.  This simplification occurs in many other nonlinear dispersive equations such as  KdV type equations and even in the integrable 
cases $b=2$ and $b=3$ of the b-CH equation \eqref{e:bch}:  see, for example, the works \cite{BrJK11,BrJK14,J09,J_BBM,	GMNP22,GP22}.  In these other cases,
where the Lagrange multipliers do not depend on several traveling wave parameters, the condition in Theorem \ref{T:morse} simplifies greatly.  
For example, for generalized KdV
type equations one has $\omega_1=c$ and hence the number of negative eigenvalues of $\mathcal{L}[\mu]$ depends only on $T_E$, i.e. it depends on the monotonicity of the
period function with respect to the ODE energy.  There is a long history of studying period monotonicity with respect to such single parameters: see, for example,
\cite{Chicone87,Schaaf85}.  The present situation is more complicated due to the fact that the Lagrange multipliers depend on several traveling wave parameters.
Nevertheless, the sign of $\{T,\omega\}_{E,c}$ can be determined for some waves in the integrable cases $b=2$ and $b=3$: see Remark \ref{R:morse} and Appendix
\ref{S:A1}.
\end{rem}

Throughout the remainder of our stability analysis, we will assume that\footnote{We note again that, as discussed in Remark \ref{R:morse}, this condition
can be analytically verified for some waves in the integrable cases $b=2$ and $b=3$.}
\begin{equation}\label{e:assume1}
\{T,\omega_1\}_{E,c}>0.
\end{equation}
From Theorem \ref{T:morse} it follows that our given $T$-periodic traveling wave solution $\mu$ is a degenerate saddle point of the functional $\Lambda$
on $L^2_{\rm per}(0,T)$, with one unstable direction and one neutral direction.  To accommodate these negative and null directions, we note that the
evolution of \eqref{e:bch_m} does not occur on all of $L^2_{\rm per}(0,T)$ but rather on the co-dimension two submanifold
\[
\Sigma_0:=\left\{m\in H^1_{\rm per}(0,T):F_1(m)=F_1(\mu(\cdot;a,E,c)),~~F_2(m)=F_2(\mu(\cdot;a,E,c))\right\}.
\]
Further, since the evolution of \eqref{e:bch_m} remains invariant under the one-parameter group of isometries corresponding to spatial translations,
this motivates us to define the group orbit of $\mu\in\Sigma_0$ as
\[
\mathcal{O}_\mu:=\left\{\mu(\cdot-x_0):x_0\in\RM\right\}.
\]
We note that $\mathcal{O}_\mu\subset\Sigma_0$ and that a solution of \eqref{e:bch_m} with initial data in $\Sigma_0$ will remain in $\Sigma_0$
for all future times.  Our strategy moving forward will be to demonstrate the ``convexity" of the functional $\Lambda$ in \eqref{e:LambdaOmega}
on the nonlinear manifold $\Sigma_0$ in a neighborhood of the group orbit $\mathcal{O}_\mu$.

To this end, we define
\[
\mathcal{T}_0:=\left\{m\in H^1_{\rm per}(0,T):\left<m,\frac{\partial F_1}{\partial m}(\mu)\right>=\left<m,\frac{\partial F_2}{\partial m}(\mu)\right>=0\right\}
\]
and note that $\mathcal{T}_0$ is precisely the tangent space in $H^1_{\rm per}(0,T)$ to the submanifold $\Sigma_0$ at the point $\mu$.  
Our next result establishes the positivity of the linear operator $\mathcal{L}[\mu]$ on the subspace of the tangent space $\mathcal{T}_0$  that is orthogonal
to the kernel of $\mathcal{L}[\mu]$.

\begin{lem}\label{L:coercive_L2}
Assume that $\{T,\omega_1\}_{E,c}>0$ and that the product $\{T,F_1\}_{E,c}\{T,F_1,F_2\}_{a,E,c}$ is negative.  Then
\[
\inf\left\{\left<\mathcal{L}[\mu]m,m\right>: \|m\|_{L^2_{\rm per}(0,T)}=1,~~m\in\mathcal{T}_0,~~m\perp \mu_x\right\}>0.
\]
In particular, there exists a constant $C>0$ such that
\[
\left<\mathcal{L}[\mu]m,m\right>\geq C\|m\|_{L^2(0,T)}^2
\]
for all $m\in\mathcal{T}_0$ with $m\perp\mu_x$.
\end{lem}

\begin{proof}
We define the function
\[
\psi=\{\mu,T,F_1\}_{a,E,c}
\]
and note that $\psi$ is smooth, $T$-periodic and satisfies
\begin{align*}
\mathcal{L}[\mu]\psi
&= \mdet{(\omega_1)_a \frac{\partial F_1}{\partial m} + (\omega_2)_a \frac{\partial F_2}{\partial m}
& (\omega_1)_E \frac{\partial F_1}{\partial m}
& (\omega_1)_c \frac{\partial F_1}{\partial m} \\
T_a  & T_E & T_c \\
(F_1)_a  & (F_1)_E & (F_1)_c}\\
&= \{\omega_1,T,F_1\}_{a,E,c} \frac{\partial F_1}{\partial m}
+ (\omega_2)_a  \{T,F_1\}_{E,c} \frac{\partial F_2}{\partial m}.
\end{align*}
In particular, we have that $\left<\mathcal{L}[\mu]\psi,m\right>=0$ for all $m\in\mathcal{T}_0$.  Further, we have
\begin{align*}
\left<\mathcal{L}\psi,\psi\right>&=
\{\omega_1,T,F_1\}_{a,E,c}\left<\frac{\partial F_1}{\partial m},\{\mu, T, F_1\}_{a,E,c}\right>\\
&\quad 
+ (\omega_2)_a\{T,F_1\}_{E,c}\left<\frac{\partial F_2}{\partial m},
\{\mu,T,F_1\}_{a,E,c}\right>
\end{align*}
Noting that
\[
\int_0^T\frac{\partial F_j}{\partial m}\nabla_{(a,E,c)}\mu~dx=\nabla_{(a,E,c)}F_j-\mu(T)\nabla_{(a,E,c)}T
\]
for both $j=1,2$, we see that
\[
\left<\frac{\partial F_1}{\partial m},\{\mu, T, F_1\}_{a,E,c}\right>=\{F_1,T,F_1\}_{a,E,c}=0
\]
and
\[
\left<\frac{\partial F_2}{\partial m},\{\mu, T, F_1\}_{a,E,c}\right>=\{F_2,T,F_1\}_{a,E,c}=\{T,F_1,F_2\}_{a,E,c},
\]
and hence
\[
\left<\mathcal{L}[\mu]\psi,\psi\right>= (\omega_2)_a\{T,F_1\}_{E,c}\{T,F_1,F_2\}_{a,E,c}.
\]
Since $(\omega_2)_a>0$ (see Remark \ref{r:LagrangeMult}), it follows by assumption that the quantity $\left<\mathcal{L}[\mu]\psi,\psi\right>$ is negative.  

Now, let $m\in\mathcal{T}_0$ be such that $\left<m,\mu_x\right>=0$.  Due to Theorem \ref{T:morse},  the assumption $\{T,\omega_1\}_{E,c}>0$ implies we can 
write $\psi=\alpha\chi+\beta\mu_x+p$ and $m=A\chi+\widetilde{p}$ for some constants $\alpha,\beta,A\in\RM$, the function $\chi$ belongs to the negative invariant space for 
$\mathcal{L}[\mu]$ and functions $p$ and $\widetilde{p}$ belong to the positive invariant subspace of $\mathcal{L}[\mu]$.  It follows then that
\begin{equation}\label{e:want1}
\left<\mathcal{L}[\mu]m,m\right>=-\lambda^2 A^2+\left<\mathcal{L}[\mu]\widetilde{p},\widetilde{p}\right>,
\end{equation}
where here $-\lambda^2<0$ is the  negative eigenvalue associated to $\chi$.  Further, we have
\[
0=\left<\mathcal{L}[\mu]\psi,m\right>=-\lambda^2A\alpha+\left<\mathcal{L}[\mu]p,\widetilde{p}\right>
\]
and
\[
0>\left<\mathcal{L}[\mu]\psi,\psi\right>=-\lambda^2\alpha^2+\left<\mathcal{L}[\mu]p,p\right>
\]
and hence, noting that the bilinear form $\left<\mathcal{L}[\mu]\cdot,\cdot\right>$ is an inner-product on the positive invariant subspace of $\mathcal{L}[\mu]$,
an application of Cauchy-Schwarz gives
\[
\left<\mathcal{L}[\mu]\widetilde{p},\widetilde{p}\right>\geq \frac{\left<\mathcal{L}[\mu]\widetilde{p},p\right>^2}{\left<\mathcal{L}[\mu]p,p\right>}>-\lambda^2 A^2.
\]
Substituting this bound into \eqref{e:want1}, the result now follows.
\end{proof}

\begin{rem}\label{R:whitham}
It is important to note that the Jacobian determinants arising above have important physical significance in their relation to the so-called Whitham theory
of modulations, which aims to study the stability of long wavelength perturbations that affect the continuous Lie symmetries of the underlying wave, i.e.
to slow modulations of the periodic traveling wave.  Whitham's theory of modulations is a well developed physical theory for dealing with such problems, consisting
of formal WKB/averaging approaches to derive quasilinear systems of PDEs, often referred to as the Whitham system,
to describe the slow evolution of the frequency and conserved quantities of wave when subjected to such long-wavelength (modulational) perturbations:
for more information, see \cite{Whitham65,Whitham74}.
 As such, from the point of view
of Whitham's theory of modulations, the period and conserved quantities, which for the bCH equation under consideration are $T$, $F_1$, and $F_2$, 
provide a natural set of coordinates for the underlying manifold of periodic traveling waves.  Of course, from our work in Section \ref{S:exist} we
know that the parameters $(a,E,c)$ provide a smooth parameterization of the manifold of periodic traveling wave solutions of \eqref{e:bch}.
In order for these two parameterizations to be compatible 
one needs to ensure the ability to smoothly change between these coordinate systems, which is precisely guaranteed by
the non-vanishing of the Jacobian determinant
\[
\{T,F_1,F_2\}_{a,E,c} = \det\left(\frac{\partial(T,F_1,F_2)}{\partial(a,E,c)}\right).
\]
Similarly, the non-vanishing of $\{T,F_1\}_{E,c}$ is equivalent to requiring that $T$ and $F_1$ provide 
a smooth parameterization of the family of traveling waves with a fixed value of the parameter $a$.  As such, at least the non-vanishing
of the product $\{T,F_1\}_{E,c}\{T,F_1,F_2\}_{a,E,c}$ can be seen to be a natural geometric requirement on the manifold of periodic traveling waves,
at least from the standpoint of Whitham's theory of modulations.  In fact, it has been shown in several models that this requirement allows
for the rigorous justification (at the level of spectral stability) of the stability predictions coming from Whitham's theory: see, for example,
\cite{BrHJ16_2,CMP22,JP20,JZBR10} and references therein.
\end{rem}

Next, we upgrade the result of Lemma \ref{L:coercive_L2} to provide a coercivity bound in $H^1$.

\begin{prop}\label{P:coercive_H1}
Under the same hypotheses as Lemma \ref{L:coercive_L2}, we have
\[
\left<\mathcal{L}[\mu]m,m\right>\geq C\|m\|_{H^1(0,T)}^2
\]
for all $m\in\mathcal{T}_0$ with $m\perp\mu_x$.
\end{prop}

\begin{proof}
This follows by an elementary interpolation argument.  Indeed, recalling \eqref{e:L} and rewriting $\mathcal{L}[\mu]$ in the symmetric form
\[
\mathcal{L}[\mu]=-\partial_x\left(p(x)\partial_x\right)+r(x)
\]
we note that for $m\in\mathcal{T}_0$ with $m\perp\mu_x$ we have
\begin{align*}
\left<\mathcal{L}[\mu]m,m\right>&=\int_0^T p\left(m_x\right)^2dx+\int_0^T r m^2~dx\\
&\geq\alpha\int_0^T(m_x)^2dx + \beta\int_0^T m^2dx,
\end{align*}
where here
\[
\alpha:=\inf_{0\leq x\leq T}p(x)~~{\rm and}~~\beta=\inf_{0\leq x\leq T}r(x).
\]
Note, specifically, that $\alpha>0$ by \eqref{e:L_coeff}.  Further, from Lemma \ref{L:coercive_L2} we know there exists a constant
$C_1>0$ such that
\[
\left<\mathcal{L}[\mu]m,m\right>\geq C_1\|m\|_{L^2(0,T)}^2
\]
for all such $m$.  Interpolating these bounds we find that
\[
\left<\mathcal{L}[\mu]m,m\right>\geq 
\alpha \gamma \int_0^T \left(m_x\right)^2 \, dx 
+ \left[\beta \gamma + (1-\gamma)C_1 \right] \int_0^T m^2 \, dx,
\]
where here $\gamma\in[0,1]$ is arbitrary.  Since $\alpha,C_1>0$, the result now follows follows by fixing $\gamma>0$ sufficiently small so that
\[
\beta\gamma+(1-\gamma)C_1>0.
\]
\end{proof}

To proceed, we introduce the semidistance $\rho:H^1_{\rm per}(0,T)\to\RM$ given by
\[
\rho(m_1,m_2)=\inf_{x_0\in\RM}\left\|m_1-m_2(\cdot-x_0)\right\|_{H^1(0,T)}
\]
and note that for a given $m\in H^1_{\rm per}(0,T)$, $\rho(m,\mu)={\rm dist}\left(m,\mathcal{O}_\mu\right)$.  We now show that
the functional $\Lambda$ in \eqref{e:L} is coercive on the nonlinear submanifold $\Sigma_0$ near the periodic
traveling wave $\mu$.

\begin{prop}\label{P:coercivity_nonlinear}
Under the hypothesis of Lemma \ref{L:coercive_L2} there exist a $\delta>0$ and a constant $C=C(\delta)>0$
such that if $m\in\Sigma_0$ with $\rho(m,\mu)<\delta$ then 
\begin{equation}\label{e:coercive_nonlinear}
\Lambda(m)-\Lambda(\mu)\geq C\rho(m,\mu)^2.
\end{equation}
\end{prop}

\begin{proof}
First, we note that the Implicit Function Theorem implies that for $\delta>0$ sufficiently small and for a $\delta$-neighborhood 
$\mathcal{U}_\delta=\left\{m\in H^1_{\rm per}(0,T):\rho(m,\mu)<\delta\right\}$ of the group orbit $\mathcal{O}_\mu$, there exists a 
unique $C^1$ map $\gamma:\mathcal{U}_\delta\to\RM$ such that
\[
\gamma(\mu)=0~~~{\rm and}~~~\left<m\left(\cdot+\gamma(m)\right),\mu_x\right>=0
\]
for all $m\in\mathcal{U}_\delta$.  Note that since $\Lambda$ is invariant under spatial translations, it suffices to establish
\eqref{e:coercive_nonlinear} with $m$ replaced by $m(\cdot+\gamma(m))$.  Now, fix $m\in\mathcal{U}_\delta\cap\mathcal{T}_0$
and note that we have the decomposition
\begin{equation}\label{e:decomp1}
m(\cdot+\gamma(m))=\mu + C_1\frac{\partial F_1}{\partial m}(\mu)
	+\left(C_2-C_1\frac{\left<\frac{\partial F_1}{\partial m}(\mu),\frac{\partial F_2}{\partial m}(\mu)\right>}{\left<\frac{\partial F_2}{\partial m}(\mu),\frac{\partial F_2}{\partial m}(\mu)\right>}\right)\frac{\partial F_2}{\partial m}(\mu)+y
\end{equation}
where $C_1,C_2\in\RM$ and $y\in\mathcal{T}_0$.  Notice that if $m=\mu$ then we clearly have $C_1=C_2=y=0$.  We thus expect these quantities to be small
for $m\in\mathcal{U}_\delta$.

To quantify this, let $v=m(\cdot+\gamma(m))-\mu$ and note that, by possibly replacing $\mu$ by an appropriate spatial translate, we may assume $\|v\|_{H^1}<\delta$.
As $F_1$ and $F_2$ are invariant with respect to spatial translation, it follows by Taylor's theorem that
\begin{equation}\label{e:F1F2_expand_near}
\left\{\begin{aligned}
F_1(m)&=F_1\left(m(\cdot+\gamma(m))\right)=F_1(\mu)+\left<\frac{\partial F_1}{\partial m}(\mu),v\right>+\mathcal{O}\left(\|v\|_{H^1}^2\right)\\
F_2(m)&=F_2\left(m(\cdot+\gamma(m))\right)=F_2(\mu)+\left<\frac{\partial F_2}{\partial m}(\mu),v\right>+\mathcal{O}\left(\|v\|_{H^1}^2\right)
\end{aligned}\right.
\end{equation}
Noting from the decomposition \eqref{e:decomp1} that 
\[
\left<\frac{\partial F_2}{\partial m}(\mu),v\right>=C_2\left<\frac{\partial F_2}{\partial m}(\mu),\frac{\partial F_2}{\partial m}(\mu)\right>
\]
it follows from \eqref{e:F1F2_expand_near}(ii) above, noting specifically that $F_2(m)=F_2(\mu)$, that $C_2=\mathcal{O}\left(\|v\|_{H^1}^2\right)$.  Similarly,
\begin{align*}
\left<\frac{\partial F_1}{\partial m}(\mu),v\right>=C_1\left(\left<\frac{\partial F_1}{\partial m}(\mu),\frac{\partial F_1}{\partial m}(\mu)\right>
		-\frac{\left<\frac{\partial F_1}{\partial m}(\mu),\frac{\partial F_2}{\partial m}(\mu)\right>^2}{\left<\frac{\partial F_2}{\partial m}(\mu),\frac{\partial F_2}{\partial m}(\mu)\right>}\right)
		+C_2\left<\frac{\partial F_1}{\partial m}(\mu),\frac{\partial F_2}{\partial m}(\mu)\right>
\end{align*}
and hence, since the Cauchy-Schwarz implies the coefficient of $C_1$ is non-zero, we infer also that $C_1=\mathcal{O}\left(\|v\|_{H^1}^2\right)$.

Now, since $\mu$ is a critical point of the action functional $\Lambda(m)$ in \eqref{e:LambdaOmega}, which is of invariant under spatial translations, Taylor's
theorem further implies that
\[
\Lambda(m)=\Lambda\left(m\left(\cdot+\gamma(m)\right)\right)=\Lambda(\mu)+\frac{1}{2}\left<\frac{\partial^2 \Lambda}{\partial m^2}(\mu)v,v\right>+o\left(\|v\|_{H^1}^2\right)
\]
and hence, using the decomposition \eqref{e:decomp1} along with the above estimates on the constants $C_{1,2}$, we find
\[
\Lambda(m)-\Lambda(\mu)=\frac{1}{2}\left<\mathcal{L}[\mu]v,v\right>+o\left(\|v\|_{H^1}^2\right)=
	\frac{1}{2}\left<\mathcal{L}[\mu]y,y\right>+\mathcal{O}\left(\|v\|_{H^1}^2\right).
\]
Since $y\in\mathcal{T}_0\cap\{\mu_x\}^\perp$, it follows by Proposition \ref{P:coercive_H1} that
\[
\left<\mathcal{L}[\mu]y,y\right>\geq C\|y\|_{H^1}^2.
\]
Noting that the reverse triangle inequality gives
\begin{align*}
\|y\|_{H^1}&\geq\left| \|v\|_{H^1}-\left\| C_1\frac{\partial F_1}{\partial m}(\mu)
	+\left(C_2-C_1\frac{\left<\frac{\partial F_1}{\partial m}(\mu),\frac{\partial F_2}{\partial m}(\mu)\right>}{\left<\frac{\partial F_2}{\partial m}(\mu),\frac{\partial F_2}{\partial m}(\mu)\right>}\right)\frac{\partial F_2}{\partial m}(\mu)\right\|_{H^1}\right|\\
	&\geq \|v\|_{H^1}-\widetilde{C}\|v\|_{H^1}^2,
\end{align*}
where here we again used the above estimates on $C_{1,2}$, it follows that
\[
\Lambda(m)-\Lambda(\mu)\geq C\|v\|_{H^1}^2=C\left\|m(\cdot+\gamma(m))-\mu\right\|_{H^1}^2\geq C\rho(m,\mu)^2,
\]
as desired.
\end{proof}

With the above preliminaries, we are ready to state and establish our main result.

\begin{thm}[Main Result]\label{T:main}
For a fixed $b>1$, let $\mu(\cdot;a,E,c)$ be a $T$-periodic solution of \eqref{e:quad}.  Assume that $\{T,\omega_1\}_{E,c}>0$ and
that the product $\{T,F_1\}_{E,c}\{T,F_1,F_2\}_{a,E,c}$ is negative.  Given any $\varepsilon>0$ sufficiently
small there exists a constant $C=C(\varepsilon)>0$ such that if $v\in H^1_{\rm per}(0,T)$ and $\|v\|_{H^1_{\rm per}(0,T)}\leq \eps$ and
if $m(\cdot,t)$ is a solution of \eqref{e:bch_m} for some interval of time with the initial condition $u(\cdot,0):=\mu+v$
then $m(\cdot,t)$ may be continued to a solution for all $t>0$ such that
\[
\sup_{t>0}\inf_{x_0\in\mathbb{R}}\left\|m(\cdot,t)-\mu(\cdot-x_0)\right\|_{H^1_{\rm per}(0,T)}\leq C\|v\|_{H^1_{\rm per}(0,T)}.
\]
\end{thm}

\begin{rem}
Recalling that a $T$-periodic solution $m(\cdot,t)$ of \eqref{e:bch_m} corresponds to a $T$-periodic solution $u(\cdot,t)=(1-\partial_x^2)^{-1}m(\cdot,t)$
of \eqref{e:bch}, the above $H^1_{\rm per}(0,T)$ stability result for the momentum density $m(\cdot,t)$ immediately translates
to an orbital stability result in $H^3_{\rm per}(0,T)$ for the associated solution $u(\cdot,t)$.
\end{rem}

\begin{proof}(Proof of Theorem \ref{T:main})  Let $\delta>0$
be such that Proposition \ref{P:coercivity_nonlinear} holds and let $v\in H^1_{\rm per}(0,T)$ satisfy
$\rho(\mu+v,\mu)\leq \eps$ for some $0<\eps<\delta$ sufficiently small.  By replacing $v$ by an appropriate spatial translate, if necessary,
we may assume that $\|v\|_{H^1}\leq\eps$.  Since $\mu$ is a critical point of $\Lambda$, Taylor's theorem implies that
$\Lambda(\mu+v)-\Lambda(\mu)\leq C\eps^2$.  Further, if $\mu+v\in\Sigma_0$ then the unique solution $m(\cdot,t)$ of \eqref{e:bch_m}
with initial condition $m(\cdot,0)=\mu+v$ must lie in $\Sigma_0$ for as along as the solution exists.  
Since $\Lambda(m(\cdot,t))=\Lambda(m(\cdot,0))=\Lambda(\mu+v)$ independently of $t$, it follows by Proposition \ref{P:coercivity_nonlinear} that
$\rho(m(\cdot,t),\mu)\leq C\eps$ for all $t\geq 0$.  This establishes Theorem \ref{T:main} in the case of perturbations which
preserve the conserved quantities $F_1$ and $F_2$.

If $\mu+v\notin\Sigma_0$, then we claim we can vary the constants $(a,E,c)$ slightly in order to effectively reduce this case to the previous one.
Indeed, notice that since we have assumed $\{T,F_1,F_2\}_{a,E,c}\neq 0$ at $\mu=\mu(\cdot;a_0,E_0,c_0)$, it follows that the map
\[
(a,E,c)\mapsto\left(T(\mu(\cdot;a,E,c)),F_1(\mu(\cdot;a,E,c)),F_2(\mu(\cdot;a,E,c))\right)
\]
is a diffeomorphism from a neighborhood of $(a_0,E_0,c_0)$ onto a neighborhood of 
\[
\left(T(\mu(\cdot;a_0,E_0,c_0)),F_1(\mu(\cdot;a_0,E_0,c_0)),F_2(\mu(\cdot;a_0,E_0,c_0))\right).
\]
In particular, we can find constants $\Delta a$, $\Delta E$ and $\Delta c$ with $|\Delta a|+|\Delta E|+|\Delta c|=\mathcal{O}(\eps)$
such that the function
\[
\widetilde\mu=\widetilde\mu\left(\cdot;a_0+\Delta a,E_0+\Delta E,c_0+\Delta c\right)
\]
is an $H^1_{\rm per}(0,T)$ solution\footnote{In particular, we choose $|(\Delta a,\Delta E,\Delta c)|\ll 1$
so that $T(a_0+\Delta a,E_0+\Delta E,c_0+\Delta c)=T(a_0,E_0,c_0)$, i.e. so that 
$\widetilde{\mu}$ and $\mu$  have the same period.} of \eqref{e:bch_m} and the constants further satisfy 
\begin{align*}
F_1(a_0+\Delta a,E_0+\Delta E,c_0+\Delta c)&=F_1(\mu+v)\\
F_2(a_0+\Delta a,E_0+\Delta E,c_0+\Delta c)&=F_2(\mu+v).
\end{align*}
Defining the augmented action functional
\[
\widetilde\Lambda(m)=E(m)-\omega_1(a_0+\Delta a,E_0+\Delta E,c_0+\Delta c)F_1(m)-\omega_2(a_0+\Delta a)F_2(m)
\]
for $m\in H^1_{\rm per}(0,T)$, where $\omega_1,\omega_2$ are defined as in Lemma \ref{L:Lagrange}, it follows as before that
\[
\widetilde\Lambda(m(\cdot,t))-\widetilde\Lambda(\widetilde\mu)\geq C_3\rho(m(\cdot,t),\widetilde\mu)^2
\]
for some $C_3>0$ as long as $\rho(m(\cdot,t),\widetilde\mu)$ is sufficiently small.  Since $\widetilde\mu$
is a critical point of the functional $\widetilde\Lambda$, we have
\[
C_3\rho(m(\cdot,t),\widetilde\mu)^2\leq \widetilde\Lambda(m(\cdot,0))-\widetilde\Lambda(\widetilde\mu)\leq C_4\left\|m(\cdot,0)-\widetilde\mu\right\|_{H^1_{\rm per}(0,T)}^2
\]
for some constant $C_4>0$.  Moreover, by the triangle inequality we have
\[
\left\|m(\cdot,0)-\widetilde\mu\right\|_{H^1_{\rm per}(0,T)}\leq\left\|m(\cdot,0)-\mu\right\|_{H^1_{\rm per}(0,T)}+\left\|\mu-\widetilde\mu\right\|_{H^1_{\rm per}(0,T)}\leq C_5\eps 
\]
for some $C_5>0$, and hence there exists a constant $C_6>0$ such that
\[
\rho(m(\cdot,t),\mu)\leq \rho(m(\cdot,t)\widetilde\mu)+\left\|\widetilde\mu-\mu\right\|_{H^1_{\rm per}(0,T)}\leq C_6\eps 
\]
for all $t>0$.  This completes the proof of Theorem \ref{T:main}.  
\end{proof}

\begin{rem}
For an alternative argument in the case when $\mu+v\notin\Sigma_0$ in the above proof, see \cite{ANP19}.
\end{rem}

Theorem \ref{T:main} provides geometric conditions guaranteeing the orbital stability of $T$-periodic traveling wave solutions of \eqref{e:bch} to perturbations which are $T$-periodic.
It thus remains to analyze the signs of the quantity $\{T,\omega_1\}_{E,c}$ as well as that of the product $\{T,F_1\}_{E,c}\{T,F_1,F_2\}_{a,E,c}$.  As discussed in 
Remark \ref{R:morse}, the Jacobian $\{T,\omega_1\}_{E,c}$ can be analytically shown to be positive for at least some waves in the integrable 
cases $b=2$ and $b=3$, corresponding to the Camassa-Holm and Degasperis-Procesi equations, respectively.  For such waves, Lemma \ref{T:morse} implies
that the Hessian operator $\mathcal{L}[\mu]$ will have only one negative eigenvalue as well as a simple eigenvalue at the origin.

To date we have not been able to analytically evaluate the remaining two Jacobians, namely $\{T,F_1\}_{E,c}$ and $\{T,F_1,F_2\}_{a,E,c}$.  As discussed in
Remark \ref{R:whitham} these quantities arise naturally in the stability analysis of periodic traveling waves to more general classes of perturbations (namely,
to long-wavelength modulational perturbations).  As such, we believe their appearance in the co-periodic stability analysis of the same waves is completely
natural and expected.  

In previous works on nonlinear dispersive equations of KdV and NLS type with polynomial type nonlinearities, these remaining Jacobians could be analytically studied through the use of the
so-called Picard-Fuchs relation: see \cite{BrJK11,LBrJM21}.  Specifically, in those examples the period and conserved quantities of the given
equation , as well as their various derivatives in terms of the ODE  parameters  $a$, $e$ and $c$, 
could be expressed as Abelian integrals on an appropriate Riemann surface.  Since there are only a finite number of such Abelian integrals it follows
that the derivatives of $T$, $F_1$ and $F_2$ with respect to $a$, $E$, and $c$ could be expressed in terms of the quantities
$T$, $F_1$ and $F_2$ themselves.  These relations, known as the Picard-Fuchs relation, allow the various Jacobians in the associated stability theory
to be calculated directly (either analytically or numerically) in terms of the underlying wave itself.  

Unfortunately, in order to use the Picard-Fuchs relation it must be that the potential  function in the ODE existence theory
is a polynomial in the wave $\phi$.  For the b-CH equation, the potential function $V(\phi;a,c)$ is given explicitly 
in \eqref{e:potential}, and for $b>1$ it is certainly not a polynomial function in $\phi$.  As such, the algebraic approaches from the 
previous works \cite{BrJK11,LBrJM21} do not apply.  

It is possible that these Jacobians can be analytically determined for general $b>1$ by considering either waves with asymptotically
small oscillations (bifurcating from the minimum $\phi_2$ of the potential $V(\phi;a,c)$) or with asymptotically large periods (taking the solitary
wave limit $0<V(\phi_1;a,c)-E\ll 1$).  

Additionally, in the integrable cases $b=2$ and $b=3$ one can determine closed form expressions for the 
periodic traveling waves, albeit with respect to different parameterizations than used in this paper: see, for example, \cite{GMNP22} for the
case $b=2$ and \cite{GP22} for the case $b=3$.  It may be possible to use these expressions to study our stability conditions numerically.
While this is currently beyond the expertise of the authors, we are hopeful these discussions will motivate other researchers to carry out the 
relevant analysis.

Finally, for general $b>1$ one may be able to numerically determine the signs of the Jacobians arising in our stability theory.  The fact that the waves 
depend on several parameters, and that the Jacobian conditions depend on derivatives of the conserved quantities with respect to all of these
parameters, this is seemingly a more delicate numerical problem than considered numerically for solitary wave studies or even for the periodic
studies in \cite{GMNP22,GP22}.  Such numerical computations are outside the scope of the author's expertise.  However, we sincerely
hope that our work motivates other researchers to pursue such a detailed numerical study.

\section{Conclusions}
In this work, we have derived a set of geometric criteria that guarantee the nonlinear orbital stability of smooth periodic traveling wave solutions to the b-CH equation \eqref{e:bch}
with $b>1$.  The stability analysis was performed in terms of the momentum variable $m$, which was shown to be controlled in $H^1_{\rm per}(0,T)$.  Correspondingly,
solutions $u(x,t)$ of \eqref{e:bch} are controlled in $H^3_{\rm per}(0,T)$.  The stability conditions are geometric in nature and relate to the parameterization of the manifold
of periodic traveling waves with respect to various coordinate systems and the Jacobians between them.  While we were able to use the previous works
\cite{GMNP22,GP22} to analytically verify one of the geometric criteria, relating to the number of negative eigenvalues of the associated Hessian operator,
we were unable to analytically verify the remaining criteria, although we argued that the conditions are natural and expected if one considers the stability
to more general classes of perturbations.  We expect there are certain asymptotic limits (the small amplitude and solitary wave limits) where an analytical study may be possible
for general $b>1$, and explicit solutions may yield a numerical approach for the integrable cases $b=2$ and $b=3$.  Further, a general numerical
investigation may be possible.  As such techniques are outside the expertise of the authors, here we simply identify these as outstanding open problems,
and we hope our work motivates other researchers to conduct these detailed studies.  Our work does, however, identify a set of geometric criteria that imply nonlienar 
orbital stability of a given periodic traveling wave.  

We also note that such geometric conditions appear naturally in the 
stability analysis of perioidc traveling wave solutions to more general classes of perturbations, specifically, to side-band or modulational perturbations: see, for example,
\cite{BrHJ16_2,BrJ10,LBrJM21} and references therein.  This allows for deep connections to Whitham's theory of wave modulations and its applications outside
of weakly nonlinear regime, and hence opens the door to further stability studies for the b-CH equation and rigorous connections (at the level of stability) of Whitham's theory.
In addition to the numerical and analytical stability studies mentioned above, we believe our work motivates a careful study of the modulational stability analysis of periodic
traveling wave solutions of the b-CH equation.  This coulud include not only a rigorous modulational stability study, but also a mathematically rigorous justification
of stability predictions coming from applying Whitham's theory of modulations to the b-CH equations.  This would  be a  very interesting direction for future study.

Finally, it would  be very interesting to connect the failure of our geometric criteria to the possible instability of the underlying periodic traveling wave.
For example, if $\{T,\omega_1\}_{E,c}<0$ we know that the operator $\mathcal{L}[\mu]$ will have two negative eigenvalues.  Does the wave
exhibit an instability, maybe spectrally, in this case?  Similarly, would the condition $\{T,F_1\}_{E,c}\{T,F_1,F_2\}_{a,E,c}>0$ imply an instability
of the background wave?  Such questions may be amenable to other analytical methods, such as Hamiltonian-Krein Index Theory\cite{KP13,BrJK14} or 
periodic Evans function techniques\cite{BrJ10}.  This also would be a very interesting direction for future investigation.

\appendix
\section{Appendix}\label{S:Appendix}

In this appendix, we first present an analytical study of the Jacobian $\{T,\omega_1\}_{E,c}$, which was seen by Theorem \ref{T:morse} to control the number of negative 
eigenvalues of the operator $\mathcal{L}[\mu]$.  We also establish a technical result that was needed in the proof of Theorem \ref{T:morse}.

\subsection{Analysis of $\{T,\omega_1\}_{E,c}$}\label{S:A1}

In Theorem \ref{T:morse} it was shown that the sign of the Jacobian $\{T,\omega_1\}_{E,c}$ determines the number of negative eigenvalues of the Hessian operator
$\mathcal{L}[\mu]$.  Throughout our stability analysis, it was important for us to assume that $\mathcal{L}[\mu]$ had only one negative eigenvalue, as well
as a simple eigenvalue at zero.  This spectral assumption is verified provided one can show that $\{T,\omega_1\}_{E,c}>0$.  While we can not determine the sign
of this quantity in general, in this section we discuss how a scaling identity may be used to better understand this quantity.

First, we note from the explicit expression for the Lagrange multiplier $\omega_1$ in Lemma \ref{L:Lagrange} we have
\[
\{T,\omega_1\}_{E,c}=c\left(\frac{b-1}{a^{1/b}}\right)\left(T_E-\frac{1}{c}T_c\right)
\]
Below, we use a scaling relation to rewrite the above and, in some special cases, express it in terms of quantities that have been studied in other works.

Now, we note that from the profile equation \eqref{e:quad} that the transformation
\[
\phi(x)=c\psi(x),~~a=c^{1+b}\alpha,~~E=c^2\beta
\]
normalizes the wave speed $c$ to unity, so that $\psi$, $\alpha$ and $\beta$ satify the same equation \eqref{e:quad} but with $c=1$.  It follows that the period function satisfies
the scaling relation
\[
T(a,E,c) = T\left(\frac{a}{c^{1+b}},\frac{E}{c^2},1\right)=T(\alpha,\beta,1).
\]
In particular, we see that
\[
T_c(a,E,c)=-\frac{(b+1)a}{c^{b+2}}T_a(\alpha,\beta,1)-\frac{2E}{c^3}T_E(\alpha,\beta,1)
\]
and
\[
T_E(a,E,c)=\frac{1}{c^2}T_E(\alpha,\beta,1).
\]
It follows that
\begin{equation}\label{e:morse_rescaled}
\{T,\omega_1\}_{E,c}=c\left(\frac{b-1}{a^{1/b}}\right)\left[\frac{2E+c^2}{c^4}T_E(\alpha,\beta,1)+\frac{(b+1)a}{c^{b+3}}T_a(\alpha,\beta,1)\right]
\end{equation}
Regarding the coefficient of $T_E$ above, note that from the existence theory in Section \ref{S:exist} we know for $b>1$ that periodic traveling waves
only exist when $(a,E,c)\in\mathcal{B}$, which is defined in \eqref{e:ExistenceSet}.  In particular, we must have $E>V(\phi_2;a,c)$ which, using the explicit
form of the potential function $V$ in \eqref{e:potential}, implies that
\[
E>V(\phi_2;a,c)\geq-\frac{1}{2}\phi_2^2\geq-\frac{c^2}{2},
\]
where the last inequality follows by the upper-bound on $\phi_2$ in \eqref{e:maxmin}.  Thus, the coefficient of $T_E$ in the above expression
for $\{T,\omega_1\}_{E,c}$ in \eqref{e:morse_rescaled} is positive.  Similarly, it follows by definition that the coefficient of $T_a$ in \eqref{e:morse_rescaled} is positive as well.

By the above considerations, it follows that a sufficient condition to guarantee the positivity of $\{T,\omega_1\}_{E,c}$, as required in our stability analysis 
in  is that both of the derivatives $T_E(\alpha,\beta,1)$ and $T_a(\alpha,\beta,1)$ are positive.  Both of the quantities $T_E$ and $T_a$ have been analytically studied
in the integrable cases $b=2$ and $b=3$: see \cite{GMNP22,GV15} for the case $b=2$ and and \cite{GP22}
for the  case $b=3$.  In these previous studies, the authors utilized various results on period monotonicity
 in Hamiltonian systems of ODE, and the results are summarized below.

\begin{prop}\label{P:period_monotone}
For a given $b>1$, fix $(a,E,c)\in\mathcal{B}$ as in \eqref{e:ExistenceSet}.  For both the cases $b=2$, corresponding to the classical Camassa-Holm equation,
and $b=3$, corresponding to the Degasperis-Procesi equation, one has that $T_E(a,E,c)>0$ for all $(a,E,c)\in\mathcal{B}$.  Furthermore, again in the 
cases $b=2$ and $b=3$, for a fixed $c>0$ there exists a $\widetilde{E}(b)\in(V(\phi_2(b);a,c),0)$ such that
\begin{itemize}
\item[(i)] $T_a(a,E,c)>0$ if $E\in(V(\phi_2(b);a,c),\widetilde{E}(b))$,
\item[(ii)] $T(a,E,c)$ has a single maximum point in $a$ if $E\in(\widetilde{E}(b),0)$,
\item[(iii)] $T_a(a,E,c)<0$ if $E\in(0,V(\phi_1(b);a,c))$.
\end{itemize}
\end{prop}

By the representation \eqref{e:morse_rescaled}, Proposition \ref{P:period_monotone} implies that for the cases $b=2$ and $b=3$
one necessarily has
\begin{equation}\label{e:cond1}
\{T,\omega_1\}_{E,c}>0
\end{equation}
for all $(a,E,c)\in\mathcal{B}$ with $E\in(V(\phi_2(b);a,c),\widetilde{E}(b))$.  Practically speaking, we note that $V(\phi_2(b);a,c)$ is the minimum ODE energy level associated
with the existence of periodic traveling wave solutions of \eqref{e:bch}, and $\widetilde{E}(b)$ is some (necessarily negative) energy level.  Looking
at Figure \ref{f:Vplot}, it follows that for the cases $b=2$ and $b=3$ there are periodic traveling waves (those oscillating near the local minimum of the 
potential $V$) for
which the condition \eqref{e:cond1} is necessarily satisfied.  We note, of course, that one still may have \eqref{e:cond1} being satisfied for
larger values of the energy $E$ depending on of course the magnitude of the individual derivatives $T_E$ and $T_a$ in \eqref{e:morse_rescaled}.  This balance,
however, does not seem to follow from previous works, nor does it seem that the methodologies utilized there apply.

\subsection{A Technical Result}\label{S:A2}
In this appendix, we establish a technical result needed in the proof of Theorem \ref{T:morse}.

\begin{lem}\label{L:technical}
Let $\mu=\mu(\cdot;a,E,c)$ be a $T$-periodic even traveling wave solution of the b-CH equation \eqref{e:bch_m}, and let
$\mu_+(a,E,c) := \mu(0;a,E,c)$ denote the global maximum of the solution $\mu$.  Then $\mu_{xx}(0)<0$
and 
\[
\{\mu_+,\omega_1\}_{E,c} > 0
\]
\end{lem}

\begin{proof}
First, note by \eqref{e:mu_derivatives} that 
\begin{align*}
\mu''(0)&= 
 \frac{b\phi''(0)}{c-\phi(0)}\mu(0)
\end{align*}
which, since $\phi$ has a non-degenerate local maximum at $x=0$ and satisfies $\phi(x)<c$ for all $x\in\RM$, implies that $\mu_{xx}(0)<0$ as claimed.

Next, evaluating \eqref{e:relation1} at $x=0$ and differentiating with respect to the parameter $E$ we find that
\[
\frac{\partial}{\partial E}\left[\mu_+(a,E,c)\right] = \frac{d}{dE}\left[\frac{a}{\left(c-\phi(0)\right)^b}\right] =\frac{ab\phi_E(0)}{\left(c-\phi(0)\right)^{b+1}}.
\]
Similarly, we have
\[
\frac{\partial}{\partial c}(\mu_+(a,E,c))= \frac{ab\left(c-\phi(0)\right)^{b-1}\left[\phi_c(0)-1\right]}{\left(c-\phi(0)\right)^{2b}}= \frac{ab[\phi_c(0)-1]}{\left(c-\phi(0)\right)^{b+1}}
\]
and hence, using the identities in Remark \ref{r:LagrangeMult}, gives
\begin{equation}\label{e:tech2}
\{\mu_+,\omega_1\}_{E,c}=\frac{a^{1-1/b}(b-1)b}{(c-\phi(0))^{b+1}}\left(c\frac{\partial\phi_+}{\partial E}-\frac{\partial\phi_+}{\partial c}+1\right),
\end{equation}
where here we set $\phi_+(a,E,c):=\phi(0;a,E,c)$.  To continue, note evaluating \eqref{e:quad} at $x=0$ gives, after rearranging,
\[
\left(2E+\phi_+^2\right)\left(c-\phi_+^{b-1}\right)=\frac{2a}{b-1}.
\]
Differentiating with respect to $E$ and simplifying, we find
\[
\frac{\partial\phi_+}{\partial E}\left[c\phi_+-E(b-1)-(\phi_+)^2\left(\frac{b+1}{2}\right)\right] = -(c-\phi_+).
\]
Recalling Remark \ref{r:quad_alternate} we note that
\[
(c-\phi_+)\phi''(0)=c\phi_+-(b-1)E-\frac{b+1}{2}\phi_+^2
\]
and hence, since $c-\phi_+>0$ and $\phi''(0)<0$, by above, it follows that 
\[
\frac{\partial\phi_+}{\partial E}=\frac{-(c-\phi_+)}{c\phi_+-(b-1)E-\frac{b+1}{2}\phi_+^2}=-\frac{1}{\phi''(0)}
\]
Similarly, evaluating \eqref{e:relation1} at $x=0$ and differentiating with respect to the parameter $c$ gives
\[
\frac{\partial \phi_+}{\partial c}\left[c\phi_+ -(b-1)E - \frac{b+1}{2}(\phi_+)^2\right] 
= -(b-1)\left(E+\frac{1}{2}\left(\phi_+\right)^2\right).
\]
Since \eqref{e:quad} implies
\[
E+\frac{1}{2}\phi_+^2=\frac{-a}{(b-1)(c-\phi_+)^{b-1}}
\]
we have
\[
(c-\phi_+)\phi''(0)\frac{\partial \phi_+}{\partial c}=\frac{a}{(c-\phi_+)^{b-1}},~~i.e.~~
\frac{\partial \phi_+}{\partial c}=-\frac{1}{\phi''(0)}\left(\frac{a}{(c-\phi_+)^{b}}\right)=-\frac{\phi_+-\phi''(0)}{\phi''(0)}.
\]

Together, the above calculations give
\[
c\frac{\partial\phi_+}{\partial E}-\frac{\partial\phi_+}{\partial c}+1=-\frac{c}{\phi''(0)}-\frac{\phi_++\phi''(0)}{\phi''(0)}+1=-\frac{c-\phi_+}{\phi''(0)},
\]
which is strictly positive.  By \eqref{e:tech2}, it follows that 
\[
\{\mu_+,\omega_1\}_{E,c}>0,
\]
as claimed.

\end{proof}


\end{document}